\documentclass[11pt,a4paper]{article}

\usepackage[utf8]{inputenc}
\usepackage{xspace}

\usepackage[top=1in, bottom=1in, left=1.25in, right=1.25in]{geometry}
\usepackage{amsmath,amsfonts,amssymb,amsthm}
\usepackage[svgnames]{xcolor}
\usepackage{graphicx}
\usepackage{makeidx}
\usepackage{algorithm}
\usepackage{algorithmic}
\usepackage{multicol}
\usepackage{enumitem}
\usepackage{dsfont}
\usepackage{hyperref}
\usepackage[slantedGreek,sc]{mathpazo}
\usepackage{float}
\usepackage[section]{placeins}

\usepackage[numbers,square]{natbib} 

\usepackage{txfonts}
\usepackage{tikz}

\newtheorem{theorem}{Theorem}[section]

\newtheorem{problem}[theorem]{Problem}
\newtheorem{definition}[theorem]{Definition}

\newtheorem{corollary}[theorem]{Corollary}

\newtheorem{proposition}[theorem]{Proposition}
\newtheorem{lemma}[theorem]{Lemma}

\newtheorem{conjecture}[theorem]{Conjecture}

\def\polhk#1{\setbox0=\hbox{#1}{\ooalign{\hidewidth
			\lower1.5ex\hbox{`}\hidewidth\crcr\unhbox0}}}

\floatname{algorithm}{Algorithmus}

\usepackage{titlesec}

\date{}

\title{Concepts of signed graph coloring}

\author{Eckhard Steffen\thanks{Paderborn Center for Advanced Studies and Institute for Mathematics, Paderborn University, 
		Paderborn, 33102, Germany. Email:~es@upb.de} and Alexander Vogel \thanks{Stadionstraße 56, 70771 Leinfelden-Echterdingen, Germany, Email:~misteralexvogel@web.de} }

\begin{document}
	
	\maketitle

\begin{abstract}
This paper surveys recent development of concepts related to coloring of signed graphs. Various approaches are presented and discussed. 
\end{abstract}

\maketitle

\section{Introduction and definitions}

The majority of concepts of signed graph coloring are natural extensions and generalizations of
vertex coloring and the chromatic number of graphs. However, it turns out that there are coloring concepts 
which are equivalent for graphs but they are not equivalent for signed graphs in general. 
Consequently, there are several versions of coloring and a
corresponding chromatic number of a signed graph. 
In this paper, we give a brief overview of 
various concepts and relate some of them to each other. We will use standard terminology of graph theory
and only give some necessary definitions, some of which are less standard.  

We consider finite graphs $G$ with vertex set $V(G)$ and edge set $E(G)$.
An edge $e$ with end vertices $v$ and $w$ is also denoted by $vw$. 
If  $v=w$, then $e$ is called a loop.
The degree of $v$ in $G$, denoted by $d_G(v)$, is the number of edges incident with $v$, a loop is
counting as two edges. 
The maximum degree of $G$, denoted by $\Delta(G)$, is $\max \{d_G(v) : v \in V(G)\}$, and 
$\min \{d_G(v) : v \in V(G)\}$ is the minimum degree of $G$, 
which is denoted by $\delta(G)$. 

A graph $G$ is $k$-regular, if $d_G(v)=k$ for all $v \in V(G)$.
Let $X \subseteq V(G)$  be a set of vertices. The subgraph of $G$ induced by $X$ is denoted by $G[X]$, and the set of edges with precisely one end in $X$ is denoted by $\partial_G(X)$. A circuit is a connected 2-regular graph. For $k \geq 1$,
a circuit of length $k$ is denoted by $C_k$, where $C_1$ is a loop, and $C_2$ consists of two vertices and
two edges between them.

Let $G$ be a graph and $C$ be a set. A mapping $c: V(G) \rightarrow C$ is a coloring 
of $G$. 
If $c(u) \not = c(v)$ for all $uv \in E(G)$, then $c$ is a proper coloring of $G$.
Furthermore, if $|\{c(v): v \in V(G)\}| \leq k$, then $c$ is a $k$-coloring of $G$.  The chromatic number of $G$ is the minimum number $k$ for which there is a proper $k$-coloring of $G$. In this paper we 
study proper colorings only. Thus we will skip the term ``proper'' in the following. 

A signed graph $(G,\sigma)$ is a graph $G$ together with a function 
$\sigma : E(G) \rightarrow \{ \pm 1 \}$.
The function $\sigma$ is called a signature of $G$ and $\sigma(e)$ is called the 
sign of $e$. An edge $e$ is negative if $\sigma(e) = -1$ and
it is positive otherwise. 
The set of negative edges is denoted by $N_{\sigma}$, and 
$E(G) - N_{\sigma}$ is the set of positive edges, which is also denoted by $E^+_{\sigma}(G)$.
The graph $G$ is sometimes called the underlying graph of the signed graph $(G,\sigma)$.  

Let $(G',\sigma|_{E(G')})$ be a subgraph of $(G,\sigma)$. 
The sign of $(G',\sigma|_{E(G')})$ is the product of the signs
of its edges. A circuit is positive if its sign is $+1$ and negative otherwise. 
Subgraph $(G',\sigma|_{E(G')})$ is balanced if all circuits in $(G',\sigma|_{E(G')})$ are positive, otherwise it is unbalanced. Furthermore, 
negative (positive) circuits are also often 
called unbalanced (balanced) circuits. 
If $\sigma(e)=1$ for all $e \in E(G)$, then $\sigma$ is the all positive signature and 
it is denoted by $\texttt{\bf 1}$,
and if $\sigma(e)=-1$ for all $e \in E(G)$, then $\sigma$ is the all negative signature and it is denoted by $\texttt{\bf -1}$.
 
\begin{theorem} [\cite{Balance}] \label{character_balanced}
A signed graph $(G,\sigma)$ is balanced if and only if $V(G)$ can be partitioned into two sets $A$ and $B$ (possibly empty)
such that all edges of $E(G[A]) \cup E(G[B])$ are positive and all edges of $\partial_G(A)$ are negative.  
\end{theorem}
  
A switching of a signed graph $(G,\sigma)$ at a set of vertices $X$ defines a signed graph $(G,\sigma')$ which is obtained from $(G,\sigma)$ by reversing the sign of each
edge of $\partial_G(X)$; i.e.~$\sigma'(e) = - \sigma(e)$ if $e \in \partial_G(X)$ and
$\sigma'(e) = \sigma(e)$ otherwise. If $X= \{v\}$, then we also say that 
$(G,\sigma')$ is obtained from $(G,\sigma)$ by switching at $v$.
Switching defines an equivalence relation on the set of all signed graphs on $G$. 
We say that $(G,\sigma_1)$ and $(G,\sigma_2)$ are equivalent, if they can be obtained from each other by a switching at a vertex set $X$.  
We also say that $\sigma_1$ and $\sigma_2$ are equivalent signatures of $G$, 
which is denoted by $\sigma_1 \sim \sigma_2$. 
Note that also other terms like ``resigning'' or ``re-signing'' are present 
in the literature instead of ``switching''.

\begin{theorem} [\cite{SignedGraphs}] \label{equ_classes}
Two signed graphs $(G,\sigma)$ and $(G,\sigma')$ are equivalent if and only if they have the same set of negative circuits. 
\end{theorem}

Hence, it follows with Theorem \ref{character_balanced}:

\begin{corollary}
A signed graph $(G,\sigma)$ is balanced if and only if it is equivalent to $(G,\texttt{\bf 1})$. 
\end{corollary}

We define a signed graph $(G,\sigma)$ to be antibalanced if it is equivalent to $(G,\texttt{\bf -1})$.
Clearly, $(G,\sigma)$ is antibalanced if and only if the sign product of every even circuit is 1 and it is -1 for every odd circuit.   
Note, that a balanced bipartite graph is also antibalanced.

Zaslavsky proposed in \cite{SignedGraphs} that there is a canonical form to any given switching class on a graph $G$ 
with respect to a maximal forest on $G$:

\begin{proposition} [\cite{SignedGraphs}]
Let $G$ be a graph and $T$ a maximal forest. Each equivalence class of the set of signed graphs on $G$ has a unique representative
whose edges are positive on $T$.
\end{proposition}

Hence, the number of non-equivalent signatures on any given finite, loopless graph is easy to determine.  

\begin{proposition} [\cite{Homomorphisms}] \label{number_non_equivalent}
If $G$ is a loopless graph with $m$ edges, $n$ vertices and $c$ components, then there are $2^{(m-n+c)}$ non-equivalent signatures on $G$.  
\end{proposition}

In particular, by Proposition \ref{number_non_equivalent}, there is only one such class on any given forest. 
By Theorem \ref{equ_classes}, each signature $\sigma$ defines an equivalence class on the set of all signed graphs on $G$.  
To avoid overloading papers technically, many authors do not use different notations for the equivalence class and for a representative of this class. In most cases, this does not cause any problems if characteristics of signed graphs are studied, which are invariant under switching. We will follow this approach in this paper. 

As far as we know, Cartwright and Harary \cite{C_H_1968} were the first to consider the question of signed graph coloring. 
In Section \ref{Cartwright_Harary}, we shortly introduce this concept from 1968, which seems to be motivated by Theorem \ref{character_balanced}.

The following two statements are 
natural conditions for a coloring $c$ and a corresponding chromatic number of a signed graph $(G,\sigma)$: 

\renewcommand{\labelenumi}{(\roman{enumi})}
\begin{enumerate}
\item [(A)] $c(v) \not= \sigma(e) c(w)$, for each edge $e=vw$. 
\item [(B)] Equivalent signed graphs have the same chromatic number. 
\end{enumerate}

Note that condition (A) requires an (algebraic) interpretation of the edge signs in the set of colors. 
In the 1980s, Zaslavsky introduced a concept of signed graph coloring which satisfies these two conditions. 
In Section \ref{Zaslavsky} we summarize some results of his work on signed graph coloring.

Coloring concepts of signed graphs which satisfy the conditions (A), (B) and the following 
condition (C) will be called {\em strong}.

\renewcommand{\labelenumi}{(\roman{enumi})}
\begin{enumerate}
\item [(C)] The corresponding chromatic number of a balanced graph signed graph $(G,\texttt{\bf 1})$
is equal to the chromatic number of $G$.    
\end{enumerate}

In Section \ref{MRS_KS} we display strong coloring concepts for signed graphs
which had been introduced by M\'a\v{c}ajov\'a, Raspaud, and \v{S}koviera \cite{ModIntro} 
and Kang and Steffen \cite{Circular}. 

Section \ref{Zhu} generalizes some approaches of previous sections by considering permutations on the
edges instead of signatures. This work is mainly driven by questions on coloring planar signed graphs.

A graph $G$ has a $k$-coloring if and only if there is a homomorphism from $G$ into
the complete graph on $k$ vertices. Naserasr, Rollov\'a and Sopena \cite{Homomorphisms}
study signed graph coloring from the viewpoint of signed graph homomorphisms. 
In Section \ref{NRS_Hom} we display this approach. 

The study of vertex colorings of graphs can be reduced to simple graphs. This is not the case in the context of signed graphs. 
We will use the following definition occasionally.  
For a loopless graph $G$ let $\pm G$ be the signed multigraph obtained from $G$ by replacing each edge by two edges, one positive and one negative. The multigraph
$\pm G$ is also called the signed expansion of $G$.

\section{A first approach} \label{Cartwright_Harary}

In 1968, Cartwright and Harary \cite{C_H_1968} gave the following definition of a coloring of signed graphs. 

\begin{definition} [\cite{C_H_1968}]
A $k$-coloring of a signed graph $(G,\sigma)$ is a partition of $V(G)$ into $k$ subsets (called color sets) such that for every edge $e$
with end vertices $v$ and $w$: 
\renewcommand{\labelenumi}{(\roman{enumi})}
\begin{enumerate}
\item if $\sigma(e) = -1$, then $v$ and $w$ are in different color sets,
\item if $\sigma(e) = 1$, then $v$ and $w$ are in the same color set.
\end{enumerate} 
\end{definition}

We say a signed graph has a coloring if it has a $k$-coloring for some $k > 0$. 

\begin{theorem} [\cite{C_H_1968}] \label{CH_1968}
The following statements are equivalent for a signed graph $(G,\sigma)$:
\renewcommand{\labelenumi}{(\roman{enumi})}
\begin{enumerate}
\item $(G,\sigma)$ has a coloring.
\item $(G,\sigma)$ has no negative edge joining two vertices of a positive component.
\item $(G,\sigma)$ has no circuit with exactly one negative edge. 
\end{enumerate}
\end{theorem}

Hence, a signed complete graph has a coloring if and only if it has no triangle with exactly one negative edge. 
Cartwright and Harary \cite{C_H_1968} observed that 
$(G,\sigma)$ has a 2-coloring if and only if $(G,\sigma)$ is balanced. For the all-negative 
signed graph $(G,\texttt{\bf -1})$, Theorem \ref{CH_1968} implies a classical result of K\"onig \cite{Koenig_1950}, 
that a graph is bipartite
if and only if it does not contain an odd circuit. Cartwright and Harary studied further variants of these kind of colorings.

Bezhad, Chartrand \cite{Behzad_Chartrand_1969} gave a definition of a signed line graph of a signed graph
and extended this coloring concept to edge-coloring of signed graphs.

\section{The fundamental approach} \label{Zaslavsky}

Zaslavsky's papers \cite{chromaticInvariants, Signedgraphcoloring,SignedGraphs,SignedGraphs_e,Colorful} in the early 1980s 
can be considered as pioneering work on signed graph coloring. The natural constraints for a coloring $c$ of a signed graph $(G,\sigma)$ are (A) and (B). Recall: 

\renewcommand{\labelenumi}{(\roman{enumi})}
\begin{enumerate}
\item[(A)] $c(v) \not= \sigma(e) c(w)$, for each edge $e=vw$.
\item[(B)] Equivalent signed graphs have the same chromatic number. 
\end{enumerate}

Condition (B) implies that colors have to be changed under switching.

 Let $M_{2k+1} = \{0,\pm 1,\pm 2, \dots ,\pm k\} \subseteq \mathbb{Z}$. 
Zaslavsky \cite{Signedgraphcoloring, Colorful} defined a coloring with $k$ colors or with $2k+1$ 
``signed colors'' of a signed graph $(G,\sigma)$ as a mapping $c$ from $V(G)$ 
into $M_{2k+1}$. Coloring $c$ is proper if it satisfies condition $(A)$.
As already mentioned, we only consider proper colorings, so we skip the term 
proper in the following. 
It is easy to see, that if $c$ is a coloring of $(G,\sigma)$ and $(G,\sigma')$ is obtained from $(G,\sigma)$ by switching at $X$, then $c'$ with $c'(v)= -c(v)$ if $v \in X$
and $c'(w) = c(w)$ if $w \in V(G)- X$ is a coloring of $(G,\sigma')$.

A coloring is zero-free if it does not use the color $0$. The exceptional role of the
color $0$ is due to the fact that it is self-inverse. That is, if $c$ is a coloring of
a signed graph $(G,\sigma)$, which uses the color 0, then $G[c^{-1}(0)]$ is an independent set in $G$
while for $t \not= 0$, $G[c^{-1}(t)]$ may contain negative edges. Zaslavsky \cite{chromaticInvariants,Signedgraphcoloring,Colorful} defined 
the chromatic polynomial $\chi_G(\lambda)$ to be the function for odd positive numbers $\lambda = 2k+1$ whose value 
is equal to the number of proper colorings of a signed graph with $k$ colors. For even positive numbers $\lambda = 2k$,
he defined the balanced chromatic polynomial $\chi_G^b(\lambda)$ whose value is the number of zero-free proper colorings of $(G,\sigma)$ with $k$ colors. Consequently, his definition for the 
chromatic number $\gamma(G,\sigma)$ of a signed graph $(G,\sigma)$ is the smallest $k \geq 0$ for which 
$\chi_G(2k+1) > 0$, and the zero-free chromatic number $\gamma^*(G,\sigma)$ is the smallest number $k$
for which $\chi_G^b(2k) > 0$. Clearly, if $(G,\sigma)$ has a $k$-coloring, then it has a zero-free $(k+1)$-coloring. Furthermore, equivalent signed graphs have the same 
chromatic number and the same zero-free chromatic number and therefore, condition $(B)$ is satisfied for both concepts. 

Major parts of Zalavsky's work on signed graph coloring are devoted to the interplay between colorings and zero-free colorings through the chromatic polynomial. 
Due to the large number of interesting results in this field we refer the interested reader to the original papers \cite{chromaticInvariants, Signedgraphcoloring,SignedGraphs,SignedGraphs_e,Colorful}
in this respect and focus on results on (upper) bounds for $\gamma$ and $\gamma^*$. 

\begin{theorem} [\cite{Colorful}]  \label{Zaslavsky_zero-free}
The zero-free chromatic number of a signed graph $(G,\sigma)$ is equal to
the minimum number of antibalanced sets into which $V(G)$ can be partitioned, and to
$\min\{\chi(G[E^+_{\sigma'}(G)]): \sigma' \sim \sigma \}$. 
\end{theorem} 

\begin{corollary} \label{Z_disadvantage}
Let $G$ be a graph. If $\chi(G)$ is even, then 
$\gamma(G,\texttt{\bf 1)} = \frac{1}{2} \chi(G) = \gamma^*(G,\texttt{\bf 1})$.
If $\chi(G)$ is odd, then $\gamma(G,\texttt{\bf 1)} = \frac{1}{2} (\chi(G) - 1)$ and 
$\gamma^*(G,\texttt{\bf 1}) = \frac{1}{2} (\chi(G) + 1)$.
\end{corollary} 

This gives us a first set of general upper bounds for the zero-free chromatic number. In fact, from these and the above mentioned properties one can derive the following classification.

\begin {corollary} [\cite{Colorful}] \label{classification}
Let $n$ be a positive integer, $G$ be a graph with no loops and $|V(G)|=n$.
Then $\gamma^*(G,\sigma)=n$ if $(G,\sigma)= \pm K_n$; 
$\gamma^*(G,\sigma)=n-1$ if $(G,\sigma)= \pm K_n - E'$ where $E' \subseteq E(\pm K_n)$ is either a non-empty  set of edges at one vertex of $\pm K_n$ or an unbalanced triangle; and otherwise $\gamma^*(G,\sigma)\leq n-2$, 
and $\gamma^*(G,\sigma)=1$ if and only if $(G,\sigma)$ is antibalanced.
\end{corollary}

Furthermore, upper bounds for the zero-free chromatic number in terms of the order of a graph are given.

\begin{theorem} [\cite{Colorful}]
Let $(G,\sigma)$ be a signed simple graph. If $|V(G)|=n$, then  
$\gamma^*(G,\sigma) \leq \lceil \frac{n}{2} \rceil$, with equality precisely when 
$(G,\sigma) = (K_n,\texttt{\bf 1})$ or \\
$n$ is even and $(G,\sigma)$ contains a $(K_{n-1},\texttt{\bf 1})$ or \\
$n=4$ and $(G,\sigma)$ is an unbalanced circuit on 4 vertices or \\
$n=6$ and $(G,\sigma)$ is equivalent to $(K_6,\sigma')$, where $N_{\sigma'}$ is the edge set
of a circuit of length 5. \\
Also, $\gamma^*(G,\sigma) \geq 1$, with equality precisely when $(G,\sigma)$ is equivalent to $(G,\texttt{\bf -1})$. 
\end{theorem}

A signed graph is orientation-embeddable into a surface $S$ if it is embeddable 
into $S$ and
a closed walk reverses orientation if and only if its sign product is $-1$.

\begin{theorem} [\cite{KleinBottle}] \label{Thm_Surface_Z}
Let $(G,\sigma)$ be a signed graph without positive loops. If $(G,\sigma)$ is 
orientation-embeddable
into the projective plane or into the Klein bottle, then  $\gamma(G,\sigma) \leq 2$ and $\gamma^*(G,\sigma) \leq 3$.
For both surfaces there are signed graphs where equality holds.  
\end{theorem}

\section{Strong concepts for coloring signed graphs} \label{MRS_KS}

This section considers concepts for coloring signed graphs which satisfy conditions (A), (B), and (C).

\subsection{$n$-coloring of signed graphs} \label{MRS}

For each $n \geq 1$, a set $M_n \subseteq \mathbb{Z}$ of colors is defined as $M_n = \{\pm 1,\pm 2, \dots ,\pm k\}$ if $n = 2k$, and 
$M_n = \{0,\pm 1,\pm 2, \dots ,\pm k\}$ if $n = 2k + 1$. 
An $n$-coloring of a signed graph $(G,\sigma)$ 
is a mapping $c: V(G) \longrightarrow M_n$, and $c$ is proper 
if $c(v) \not= \sigma(e) c(w)$, for each edge $e=vw$. 
The smallest number $n$ such that $(G,\sigma)$ admits a proper $n$-coloring is the signed chromatic 
number of $(G,\sigma)$ 
and it is denoted by $\chi_\pm((G,\sigma))$. Obviously, this coloring is invariant under switching
and $\chi_\pm((G,\texttt{\bf 1})) = \chi(G)$.
Hence, this coloring concept satisfies conditions (A), (B) and (C).

There is also a direct relationship between $\chi_\pm$ and the pair $\gamma$  and $\gamma^*$,
which can be expressed as 
\begin{displaymath}
\chi_\pm ((G,\sigma))= \gamma ((G,\sigma))  + \gamma^* ((G,\sigma)),
\end{displaymath}
for every signed graph $(G,\sigma)$. 

The signed chromatic number has different but often similar bounds 
as $\gamma$ and $\gamma^*$. 
For instance, 
an upper bound using the chromatic number of the underlying graph is given by the following theorems, which follow from the fact that every coloring of $G$ 
induces a coloring of $(G,\sigma)$.  

\begin{theorem} [{\cite{choosable}}]
If $G$ is a simple graph, then $\chi_\pm (\pm G) = 2\chi (G) -1$.
\end{theorem}

Since every signed graph $(G,\sigma)$ is a subgraph of $\pm G$, it follows that $2\chi (G) -1$
is an upper bound for the signed chromatic number of a signed graph $(G,\sigma)$. However, 
Máčajová, Raspaud and Škoviera proved that the bound is attained by a signed simple graph.

\begin{theorem} [{\cite{ModIntro}}] \label{MaxChrom}
If $(G,\sigma)$ is a signed simple graph, then
$\chi_\pm ((G,\sigma)) \leq 2\chi(G) -1$.
Furthermore, this bound is sharp.
\end{theorem}

Another bound can be regarded as an extension of the well known characterization of bipartite graphs.

\begin{lemma}  [{\cite{ModIntro}}]
A signed graph $(G,\sigma)$ is antibalanced if and only if $\chi_\pm ((G,\sigma)) \leq 2$. 
\end{lemma}

Important for the study of choosability of signed graphs, a topic which will be introduced later in this section, is the notion of degeneracy. A graph $G$ is called $k$-degenerate if every subgraph of $G$ has a vertex of degree at most $k$. As in the case for graphs, this 
property can be used to prove an upper bound for the signed chromatic number of a signed graph. 
If a graph is $k$-degenerate, then there is an ordering $v_1, v_2, \dots , v_n$ of its vertices such that for every $i \in \{2, \dots,n\}$ the vertex $v_i$ has at most $k$ neighbors in $\{v_1, \dots , v_{i-1}\}$. 
Now, greedy coloring yields the following lemma.

\begin{lemma}  [{\cite{ModIntro}}]
If a graph $G$ is $k$-degenerate, then $\chi_\pm((G,\sigma)) \leq k + 1$.
\end{lemma}

We recall that the vertex arboricity $a(G)$ of a graph $G$, is the minimum number of subsets into which $V(G)$ can be partitioned so that each set induces a forest. Similarly, the edge arboricity of a graph $G$, denoted by $a'(G)$, is the minimum number of forests into which its edge set can be partitioned. An acyclic coloring of a graph is a coloring in which every two color classes induce a forest and $\chi_a(G)$ denotes the acyclic chromatic number of a graph $G$. With these notions, one can describe several
upper bounds for $\chi_{\pm}$ for specific classes of signed simple graphs.

\begin{proposition} [{\cite{ModIntro}}]
If $(G,\sigma)$ is a signed simple graph, then the following statements are true. 
\renewcommand{\labelenumi}{(\roman{enumi})}
\begin{enumerate} 
\item If $G$ is $K_4$-minor-free, then $\chi_\pm((G,\sigma))\leq 3$.
\item If $G$ is the union of two forests (i.e. $a'(G)\leq 2$), then $\chi_\pm((G,\sigma))\leq 4$.
\item If $a(G)\leq k$, then $\chi_\pm((G,\sigma))\leq 2k$.
\item $\chi_\pm((G,\sigma))\leq \chi_a(G)$.
\end{enumerate}
\end{proposition}

The most fundamental result in \cite{ModIntro} is a signed version of the famous theorem of Brooks, 
which relates the signed chromatic number to the maximum degree of a graph.
The two extremal cases in the original theorem, the complete graphs and the odd circuits, carry over into the version for signed graphs as the balanced signed graphs and the balanced odd circuits. However, it is interesting that for signed graphs there is a third extremal case, the even unbalanced circuits.

\begin{theorem}[{\cite{ModIntro}}] \label{Brooks_MRS}
Let $(G,\sigma)$ be a connected signed simple graph. If $(G,\sigma)$ is not a balanced complete
graph, a balanced odd circuit, or an unbalanced even circuit, then $\chi_\pm((G,\sigma)) \leq \Delta(G)$.
\end{theorem}

It is proved in \cite{DFS} that depth-first search and greedy coloring can be used to find a proper coloring of connected signed graphs 
$(G,\sigma)$ using at most $\Delta(G)$ colors, provided $(G,\sigma)$ is different from the above mentioned extremal cases. 
Zaj{\polhk{a}}c \cite{zajc2018Short} proved a more general version of Brook's Theorem which also implies Theorem \ref{Brooks_MRS}.
Another important theorem that has its pendant in the theory of graph coloring is the five color theorem for planar signed graphs.

\begin{theorem}[{\cite{ModIntro}}] \label{MRS_Steinberg}
Let $(G,\sigma)$ be a planar signed simple graph, then $\chi_\pm((G,\sigma))\leq 5$ . Furthermore,
\renewcommand{\labelenumi}{(\roman{enumi})}
\begin{enumerate} 
\item if $G$ is triangle-free, then $\chi_\pm((G,\sigma))\leq 4$, and
\item if $G$ has girth at least 5, then $\chi_\pm((G,\sigma))\leq 3$.
\end{enumerate}
\end{theorem}

Note, that it has been conjectured in {\cite{ModIntro} that the famous 4-color theorem 
can be generalized to planar signed simple graphs.
However, this conjecture is disproved, see Section \ref{Zhu}.

Theorem \ref{MRS_Steinberg} approximates Gr\"otzsch's Theorem \cite{Groetzsch}, which states that every triangle-free planar graph is 3-colorable. 
In this context, Erd\"{o}s raised the question (see problem 9.2 in \cite{Steinberg}) whether there exists
a constant $k$ such that every planar graph without cycles of length from $4$ to $k$ is 3-colorable? 
Hu and Li have studied this question for signed graphs in \cite{Hu_Lili_2019}. 

\begin{theorem} [\cite{Hu_Lili_2019}]
Let $(G,\sigma)$ be a planar signed simple graph. If $G$ does not contain a circuit of length $k$ for all $k \in \{4,5,6,7,8\}$, 
then $\chi_{\pm}((G,\sigma)) \leq 3$.   
\end{theorem}

\subsection{$(k,d)$-coloring of signed graphs} \label{KS}

In 1988, Vince \cite{CircularOrig} introduced $(k,d)$-coloring and the now called
circular chromatic number for graphs. Originally he introduced the circular chromatic number under the name star chromatic number, see \cite{CircularEquiv} for details. The circular chromatic number of a graph $G$
is denoted by $\chi_c(G)$, and it is defined as the infimum over all rational numbers $\frac {n}{k}$ } so that there is a mapping $\varphi$ from $V(G)$ into the cyclic group of integers modulo $n$, $\mathbb {Z} /n\mathbb {Z}$ (or just $\mathbb {Z}_n$) with the property that if $u$ and $v$ are adjacent vertices in $G$, then $\varphi(u)$ and $\varphi(v)$ are at distance of at least $k$ in $\mathbb {Z}_n$. 

Kang and Steffen \cite{Circular} extended the concept of $(k,d)$-coloring to signed graphs and extended these definitions to signed graphs as follows. 

For $x \in \mathbb{R}$ and a positive real number $r$, we denote by $[x]_r$, the remainder of $x$ divided by $r$, 
and define $|x|_r = \min\{[x]_r, [-x]_r\}$. Thus, $[x]_r \in [0,r)$ and $|x|_r=|-x|_r$.
For $a,b \in \mathbb {Z}$ and an integer $k \geq 2$, $|a-b|_k$ can be regarded as the distance of $a$ 
and $b$ in $\mathbb {Z}_k$. For two
positive integers $k$ and $d$ with $k \geq 2d$, a $(k,d)$-coloring of a signed graph $(G,\sigma)$ is a mapping $c: V(G) \longrightarrow \mathbb {Z}_k$ with the property, that $d\leq|c(v)-\sigma(e)c(w)|_k$ for each edge $e=vw$. 
The circular chromatic number $\chi_c((G,\sigma))$ of a signed graph $(G,\sigma)$ is the infimum over all $\frac{k}{d}$ so that $(G,\sigma)$ has a $(k,d)$-coloring. 
In the context of integer coloring, a $(k,1)$-coloring will also be called a $\mathbb{Z}_k$-coloring or a modular $k$-coloring. The minimum $k$ such that $(G,\sigma)$ has a $(k,1)$-coloring is called the cyclic chromatic number of $(G,\sigma)$, and it is denoted by $\chi((G,\sigma))$.

Similar to $\chi_{\pm}$, we have $\chi((G,\texttt{\bf 1})) = \chi(G)$ and $\chi_c((G,\texttt{\bf 1})) = \chi_c(G)$. 
Therefore, this concept naturally generalizes $(k,d)$-coloring of graphs to signed graphs. 
Like for other definitions of signed graph coloring, the circular chromatic number is unchanged under switching. Given a signed graph $(G,\sigma)$ with a $(k,d)$-coloring $c$, a switching of $(G,\sigma)$ at $X \subseteq V(G)$ 
together with a change of $c(u)$ to $-c(u)$ for each $u \in X$
results in an equivalent graph $(G,\sigma')$ with $(k,d)$-coloring $c'$.

\begin{proposition} [{\cite{Circular}}]
Let $k$ and $d$ be positive integers, $(G,\sigma)$ be a signed simple graph and $c$ be a $(k,d)$-coloring of $(G,\sigma)$. If $(G,\sigma)$ and $(G,\sigma')$ are equivalent, then there exists a $(k,d)$-coloring $c'$ of $(G,\sigma')$. Therefore, $\chi_c((G,\sigma))=\chi_c((G,\sigma'))$.
\end{proposition}
Note, that if $(G,\sigma)$ has a $(k,d)$-coloring $c$, then there is an equivalent graph $(G,\sigma')$ with corresponding $(k,d)$-coloring $c'$ such that $c'(v)\in \{0,1,2,\dots,\left\lfloor \frac{k}{2}\right\rfloor\}$ for every $v\in V(G)$. This is easy to see by switching $(G,\sigma)$ at
every vertex $v$ with $c(v) > \left\lfloor \frac{k}{2}\right\rfloor$.

A very basic result on the circular chromatic number for graphs is that the infimum in its definition is actually a minimum. This is also true for signed graphs. 
In \cite{Circular} it is shown that the number
of used colors in a smallest $(k,d)$-coloring can be bounded by a function of the order of the graph.
Hence, the circular chromatic number is a minimum. Therefore, 
if $\chi_c((G,\sigma)) = \frac{k}{d}$, then $(G,\sigma)$ has a $(k,d)$-coloring.  

\begin{theorem} [{\cite{Circular}}]\label{infmin}
Let $(G,\sigma)$ be a signed simple graph on $n$ vertices, then 
$$\chi_c((G,\sigma)) =\text{ min} \left\{ \frac{k}{d}:(G,\sigma) \text{ has a } (k,d) \text{-coloring and } k\leq 4n \right\}.$$
\end{theorem}
For graphs, we have $\left\lceil \chi_c(G)\right\rceil =\chi(G)$, \cite{CircularOrig}. A similar result 
holds true for signed graphs as well.

\begin{theorem} [{\cite{Circular}}] \label{relation_cc}
Let $(G,\sigma)$ be a signed simple graph. Then $\chi((G,\sigma))-1\leq \chi_c((G,\sigma))\leq \chi((G,\sigma))$.
\end{theorem}

Note that, unlike for graphs, the difference between $\chi_c$ and $\chi$ can be 1. Actually, this case is equivalent to the existence of a different coloring.

\begin{theorem} [{\cite{Circular}}]
Let $(G,\sigma)$ be a signed simple graph with $\chi((G,\sigma))=t+1$ for a positive integer $t$. Then, $\chi_c((G,\sigma))=t$ if and only if $(G,\sigma)$ has a $(2t,2)$-coloring.
\end{theorem}

There is a simple construction of graphs with the aforementioned property.

\begin{theorem} [{\cite{Kang_Diss,Circular}}]
For every $k\geq 2$, there exists a signed simple graph $(G,\sigma)$ with $\chi((G,\sigma))-1=\chi_c((G,\sigma))=k$.
\end{theorem}

If, however, the difference between $\chi$ and $\chi_c$ is not 1, then this lower bound on $\chi_c$ can be further improved. Recalling Theorem \ref{infmin}, we know that the value of $\chi_c$ can be stated as $\frac{p}{q}$ with $p$ and $q$ being coprime integers and $p\leq 4n$. Thus, if $\chi_c$ and $\chi$ differ, then their difference has to be at least $\frac{1}{d}$, see \cite{Circular}.

\subsection{Circular coloring of signed graphs}

Zhu introduced circular $r$-coloring, see \cite{CircularEquiv}.  
For signed graphs, this definitions has to be modified in a similar way as 
$(k,d)$-coloring. 

Let $(G,\sigma)$ be a signed graph and $r\geq1$ be a real number. A circular-$r$-coloring of $(G,\sigma)$ is a function $f: V(G) \longrightarrow [0,r)$ such that for any edge $e=uv$ of $(G,\sigma)$, if $\sigma(e)=1$, then $1\leq |f(u) -f(v)|\leq r-1$, and if $\sigma(e)=-1$, then $1\leq |f(u) +f(v)-r|\leq r-1$. This definition can be equivalently stated if one identifies 0 and $r$ on the interval $[0,r]$, thus obtaining a circle with perimeter $r$, denoted by $S^r$. 
Now the colors are points on $S^r$ and the distance between two points $a$ and $b$ on $S^r$ is the length of the shorter arc of $S^r$ connecting $a$ and $b$, which can be described as $|a-b|_r$. Define the inverse of $a\in S^r$ to be $r-a$, then a circular $r$-coloring of $(G,\sigma)$ is a function $f: V(G) \longrightarrow [0,r)$ which satisfies the following conditions for every edge $e=uv$ of $(G,\sigma)$: 
If $\sigma(e)=1$, then $1\leq |f(u) -f(v)|_r$, and if $\sigma(e)=-1$, then $1\leq |f(u) +f(v)|_r$. This definition also enables the easy conversion under switching. If $f$ is a circular $r$-coloring of $(G,\sigma)$ and $(G,\sigma')$ is obtained by switching $(G,\sigma)$ at $X \subseteq V(G)$, then $f'$ defined as $f$ on $V(G)- X$ and as $f'=r-f$ on $X$ is a circular $r$-coloring of $(G,\sigma')$. Hence, for every circular $r$-coloring of a signed graph, 
there is a circular $r$-coloring of a switching equivalent graph that only uses colors 
in the interval $[0,\frac{r}{2}]$.

In \cite{CircularEquiv} it is shown that a graph $G$ has a $(k,d)$-coloring ($2d\leq k$) 
if and only if $G$ has circular $\frac{k}{d}$-coloring. A similar results is true for signed graphs:

\begin{theorem} [{\cite{Circular}}] \label{circ_Kang}
Let $k$ and $d$ be positive integers with $2d\leq k$. A signed simple graph $(G,\sigma)$ has a $(2k,2d)$-coloring if and only if $(G,\sigma)$ has a circular $\frac{k}{d}$-coloring. 
\end{theorem}

\begin{theorem} [{\cite{Circular}}]
Let $(G,\sigma)$ be a signed simple graph. 
Then $\chi_c((G,\sigma)) = min\{r: (G,\sigma)$ has a circular $r$-coloring$\}$.
\end{theorem}

\subsection{Relations between coloring parameters and the chromatic spectrum of a graph} \label{spectrum}

The cyclic chromatic number $\chi((G,\sigma))$ is different from the signed chromatic number $\chi_\pm(G,\sigma)$ 
as the following theorem shows. 

\begin{proposition} [{\cite{Circular}}]
If $(G,\sigma)$ is a signed simple graph, then $\chi_\pm((G,\sigma))-1 \leq\chi((G,\sigma))\leq\chi_\pm((G,\sigma))+1$.
\end{proposition}

The following proposition classifies some easy examples showing that these bounds are tight.

\begin{proposition} [{\cite{Circular}}]
\renewcommand{\labelenumi}{(\roman{enumi})}
Let $(G,\sigma)$ be a connected signed simple graph with $|V(G)|\geq 3$. 
\begin{enumerate} 
\item If $(G,\sigma)$ is antibalanced and not bipartite, then $\chi_\pm((G,\sigma))=2$ and $\chi((G,\sigma))=3 $.
\item If $(G,\sigma)$ is bipartite and not antibalanced, then $\chi_\pm((G,\sigma))=3$ and $\chi((G,\sigma))=2 $.
\end{enumerate}
\end{proposition}

Obviously, the signed chromatic number of a signed graph dependents not only on the structure of the underlying graph, but also on the signature. For example, every all-negative signed graph can be colored using color 1 only. It is therefore important to study how much the signed chromatic number of a given signed graph can be changed by replacing the corresponding signature. 

The set $\Sigma_{\chi}(G)=\{\chi((G,\sigma)): \sigma \text{ is a signature on }G \}$ is the chromatic spectrum of 
$G$ with respect to $\chi((G,\sigma))$, and let $M_{\chi} = \max \Sigma_{\chi}(G)$ and $m_{\chi} = \min \Sigma_{\chi}(G)$. Analogously, 
$\Sigma_{\chi_\pm}(G) =\{\chi_\pm((G,\sigma)): \sigma$ is a signature on $G\}$ is the chromatic spectrum of $G$ 
with respect to $\chi_{\pm}$ and $m_{\chi_\pm}(G)$ and $M_{\chi_\pm}(G)$ denote the minimum and maximum of this set, respectively.

\begin{proposition} [{\cite{Spectrum}}]
Let $G$ be a nonempty simple graph. The following statements hold.
\renewcommand{\labelenumi}{(\roman{enumi})}
\begin{enumerate} 
\item $\Sigma_{\chi}(G)=\{1\}$ $\Leftrightarrow$ $m_{\chi}=1$ $\Leftrightarrow$ $E(G)=\emptyset$.
\item If $E(G)\neq \emptyset$, then $\Sigma_{\chi}(G)=\{2\}$ if and only if $m_{\chi}=2$ if and only if $G$ is bipartite.
\item If $G$ is not bipartite, then $m_{\chi}=3$.
\item $\Sigma_{\chi_\pm}(G)=\{1\}$ if and only if $E(G)=\emptyset$.
\item If $E(G)\neq \emptyset$, then $m_{\chi_\pm}=2$.
\end{enumerate}
\end{proposition}

The third statement is obtained by using an all-negative signature on $G$ and coloring every vertex with color $1 \in \mathbb {Z}_3$. Using a similar method on $G-H$, with $H$ being an induced subgraph of $G$, the following statement is obtained.

\begin{lemma} [{\cite{Spectrum}}]
Let $k\geq 3$ be an integer. If $H$ is an induced subgraph of a simple graph $G$ with $k \in\Sigma_{\chi}(H)$
$(k \in\Sigma_{\chi_\pm}(H))$, 
then $k \in\Sigma_{\chi}(G)$ $(k \in\Sigma_{\chi_\pm}(G))$.
\end{lemma}

It turns out that the chromatic spectrum of $G$ with respect to these two coloring parameters is an interval of integers. 

\begin{theorem} [{\cite{Spectrum}}]
If $G$ is a simple graph, then $\Sigma_{\chi}(G)=\{k : k \in \mathbb{N} \text{ and } m_{\chi}(G)\leq k\leq M_{\chi}(G)\}$ and 
$\Sigma_{\chi_\pm}(G)=\{k : k \in \mathbb{N} \text{ and } m_{\chi_\pm}(G)\leq k\leq M_{\chi_\pm}(G)\}$.
\end{theorem}

We close this section with relating $\chi_{\pm}$ and $\chi_c$ to each other.

\begin{proposition} \label{circ_MRS_even}
Let $(G,\sigma)$ be a signed simple graph and $k \in \mathbb{N}$. If $(G,\sigma)$ has a $2k$-coloring, then $(G,\sigma)$ has a $(4k,2)$-coloring and a circular $2k$-coloring.
In particular, if $\chi_{\pm}((G,\sigma)) = 2k$, then $\chi_c((G,\sigma)) \leq \chi_{\pm}((G,\sigma))$. 
\end{proposition}

\begin{proof}
Let $\phi : V(G) \rightarrow \{\pm 1, \dots, \pm k\}$ be a $2k$-coloring. By possibly switching we can assume that 
$\phi(v)$ is positive for every $v \in V(G)$. Let $\phi' : V(G) \rightarrow \mathbb {Z}_{4k}$ with 
$\phi'(v) = 2 \phi(v)-1$. It is easy to see that $\phi'$ is a $(4k,2)$-coloring of $(G,\sigma)$. By Theorem \ref{circ_Kang}, $(G,\sigma)$ has a circular $2k$-coloring. The second 
part of the statement follows directly from these facts. 
\end{proof}

\subsection{Choosability on signed graphs}
 
Another important adoption from the theory of graph coloring is the notion of list-coloring and choosability for signed graphs. Both definitions of the previous two subsections ($k$-coloring and $\mathbb{Z}_k$-coloring) are used to study 
list coloring for signed graphs. 
For graphs, list coloring and choosability had been introduced by Erdős, Rubin and Taylor in \cite{unsignChoose} and many results on these parameters can be transformed into similar ones for signed graphs. The coloring number of a graph $G$, denoted by 
$col(G)$, is the maximum ranging over the minimum degree of all subgraphs of $G$ plus 1. Therefore, a graph with coloring number at most $k + 1$ is also $k$-degenerate. Note that the coloring number is unchanged under any signature assignment to $G$. 

Let $(G,\sigma)$ be a signed graph, $k\geq 0$ be an integer and $f:V(G) \longrightarrow \mathbb{N}_0$ a function. A list-assignment $L$ of $(G,\sigma)$ is a function that maps every vertex $v$ of $G$ to a nonempty set (list) of colors $L(v)\subseteq \mathbb{Z}$. More specific, if $|L(v)|= f(v)$ for every $v \in V(G)$ we call $L$ an $f$-assignment and if $|L(v)|= k$ for every $v \in V(G)$ we call it a $k$-assignment. An $L$-coloring of $(G,\sigma)$ is a proper coloring $\phi$ of $(G,\sigma)$ such that $\phi(v) \in L(v)$ for all $v \in V(G)$. If $(G,\sigma)$ admits an $L$-coloring, then $(G,\sigma)$ is said to be $L$-colorable or, more generally, list-colorable. Resulting from this, if $(G,\sigma)$ is $L$-colorable for every $f$-assignment $L$ of $(G,\sigma)$, it is called $f$-list-colorable and similarly $(G,\sigma)$ is called $k$-list-colorable or $k$-choosable if it is $L$-colorable for every $k$-assignment $L$. The signed list-chromatic number or signed choice number of $(G,\sigma)$ is the smallest integer $k \geq 0$ such that $(G,\sigma)$ is $k$-choosable. We denote it as $\chi^l_{\pm}((G,\sigma))$.

The proposition that every $(d-1)$-degenerate graph is $d$-choosable can easily be generalized via induction on the vertex set of a signed graph.
\begin{theorem}[{\cite{PlanarChoose}}] \label{choosable_Thm}
Let $(G,\sigma)$ be a signed simple graph. If $G$ is $(d - 1)$-degenerate, then $(G,\sigma)$ is $d$-choosable.
\end{theorem}

A coloring of $(G,\sigma)$ with color set $M_k$ can be regarded as an $L$-coloring for the $k$-assignment $L$ with $L(v)= M_k$ for every $v \in V(G)$. Hence, $\chi_\pm((G,\sigma)) \leq \chi^l_\pm((G,\sigma))$.

Since $\chi_\pm((G,\sigma))$ is invariant under switching, it makes sense that this also holds true for the signed choice number and some sort of switching defined for a list-assignment. In \cite{PlanarChoose} there is such a definition.

Let $(G,\sigma)$ be a signed graph, $L$ be a list-assignment of $(G,\sigma)$, and $\phi$ be an $L$-coloring of $(G,\sigma)$. Let
$X \subseteq V(G)$. We say $L'$ is obtained from $L$ by a switch at $X$ if
\begin{displaymath}
L'(v)=\left\{-\alpha : \alpha \in L(v)\right\} \text{, if }v \in X, \text {  and  }  L'(v)= L(v) \text{, if }v \in V(G)- X.
\end{displaymath}

With this definition we easily get the following observation. 

\begin{proposition} 
Let $(G,\sigma)$ be a signed graph, $L$ be a list-assignment of $G$ and $\phi$ be an $L$-coloring of $(G,\sigma)$.
If $\sigma'$, $L'$ and $\phi'$ are obtained from $\sigma$, $L$ and $\phi$ by switching at a subset of $V(G)$, then $\phi'$ is an $L'$-coloring of
$(G,\sigma')$. Furthermore, two switching equivalent signed graphs have the same signed choice number.
\end{proposition}

Schweser and Stiebitz \cite{choosable} generalized several classical coloring results 
to signed (multi-) graphs. They generalized Theorem \ref{Brooks_MRS}
and proved a list version as well as a 
degree version of this result. First they observed that
$\chi^l_\pm$ can be incorporated in a chain of inequalities that is an extension of the Brooks type formula stated earlier. 

\begin{proposition} [{\cite{choosable}}]
	Every signed graph $(G,\sigma)$ satisfies
	\begin{displaymath}
	\chi_\pm ((G,\sigma)) \leq \chi^l_\pm((G,\sigma)) \leq col((G,\sigma)) \leq \Delta(G)+1.
	\end{displaymath}
\end{proposition}

A signed graph $(G,\sigma)$ is called degree-choosable if $(G,\sigma)$ is $f$-list-colorable for the degree function $ f(v) = d_G(v)$ for all $v \in V(G)$. Not every signed graph is degree-choosable but every signed graph $(G,\sigma)$ is $f$-list-colorable if $f(v) = d_G(v) +1$ for every $v \in V(G)$. 
For the following we will recall the notion of a block 
of $G$, a maximal connected subgraph of $G$ that has no separating vertex. A
signed graph $(G,\sigma)$ is called a brick if $(G,\sigma)$ is a balanced complete graph, a balanced odd circuit, an unbalanced even circuit, a $\pm K_n$ for an integer $n \geq 2$, or a $\pm C_n$ for an odd integer $n \geq 3$. One can see that this class of signed graphs is an extension of the extremal cases for Brooks' type theorem for signed graphs. 

\begin{theorem} [{\cite{choosable}}] \label{char_Stiebitz}
Let $(G,\sigma)$ be a connected signed graph. Then $(G,\sigma)$ is not degree-choosable if and only if each block of $(G,\sigma)$ is a brick.
\end{theorem}

The following corollary is a Brooks' type theorem for the list-chromatic number of signed graphs.

\begin{corollary} [{\cite{choosable}}]
Let $(G,\sigma)$ be a connected signed graph. If $(G,\sigma)$ is not a brick, then 
$\chi^l_\pm((G,\sigma)) \leq \Delta(G)$.
\end{corollary}

With the theory of list-coloring of signed graphs, naturally the notion of list-critical signed graphs emerges.
Let $(G,\sigma)$ be a signed graph and let $L$ be a list assignment of $(G,\sigma)$. 
The signed graph $(G,\sigma)$ is called $L$-critical if $(G,\sigma)$ is not $L$-colorable, but every proper subgraph of $(G,\sigma)$ is. Particularly, if $L$ is a $(k-1)$-list-assignment, we call $(G,\sigma)$ $k$-list-critical. Also, a signed graph $(G,\sigma)$ is called $k$-critical if $\chi_\pm((G,\sigma))=k$ and for every proper subgraph $(H,\sigma')$ of $(G,\sigma)$, $\chi_\pm((H,\sigma'))\leq k-1$. With the same argument as before, we see that every $k$-critical signed graph is $k$-list-critical. A signed graph $(G,\sigma)$ is called $k$-choice-critical if $\chi^l_{\pm}((G,\sigma)) = k$ and for every proper subgraph $(H,\sigma')$ of $(G,\sigma)$ we have $\chi^l_{\pm}((G,\sigma)) \leq k-1$. Again, we get the result that every $k$-choice-critical signed graph is $k$-list-critical. 
Especially the class of 3-critical signed graphs can be characterized easily.

\begin{lemma} [{\cite{choosable}}]
A signed graph is 3-critical if and only if it is a balanced odd circuit or an unbalanced even circuit.
\end{lemma}

The following lemma states some of the basic properties of list-critical signed graphs.

\begin{lemma} [{\cite{choosable}}] 
Let $(G,\sigma)$ be an $L$-critical signed graph for a list-assignment $L$. Let $H=\{v : d_G(v) > |L(v)| \}$ and 
$F=V(G) -H$ and let $\emptyset \neq X \subseteq F$. The following statements hold true.
\renewcommand{\labelenumi}{(\roman{enumi})}
\begin{enumerate}
\item $d_G(v) = |L(v)|$ for all $v \in X$. 
\item Every block of $(G[X],\sigma|_{E(G[X])})$ is a brick.
\item If $L$ is a $(k-1)$-list-assignment with $k \geq 1$, then $ H \neq \emptyset $ or $(G,\sigma)$ is a brick. Furthermore, if $(G[X],\sigma|_{E(G[X])})$ contains a $K_k$, then $(G,\sigma)$ is a balanced complete graph of order $k$.
\end{enumerate}
\end{lemma}

This lemma implies, that for a $k$-list-critical signed graph $(G,\sigma)$ the minimum degree $\delta(G)$ is at least $k-1$ and so we have a lower bound for the number of edges in a $k$-list-critical signed graph of the form $|E((G,\sigma))| \geq \frac{1}{2}(k-1)|V(G)|$. This bound can be improved 
to $\frac{1}{2} \left(k-1+\frac{k-3}{k^2-3}\right)|V(G)|$ for certain signed simple graphs \cite{choosable}.

Major results for list coloring planar graphs can be generalized to planar signed simple 
graphs. For example, the next theorem is an extension of Thomassen's work in \cite{UnsignedPlanar} to signed graphs and its proof uses the same method.

\begin{theorem} [{\cite{PlanarChoose}}]
	Every planar signed simple graph is 5-choosable.
\end{theorem}

A result of Voigt \cite{ListColor} on the existence of not 4-choosable planar graphs
is extendable to signed graphs, too. Interestingly, there are signed graphs with this property such that their underlying graphs are 4-choosable.

\begin{theorem} [{\cite{PlanarChoose}}]
	There exists a planar signed simple graph $(G,\sigma)$ that is not 4-choosable but $G$ is
	4-choosable.
\end{theorem}

\begin{theorem} [{\cite{PlanarChoose}}]
	Let $(G,\sigma)$ be a planar signed simple graph. For each $k \in \{3,4,5,6\}$, if $(G,\sigma)$ has no $k$-circuits, 
	then $(G,\sigma)$ is 4-choosable.
\end{theorem}

Furthermore, the proof of Thomassen in \cite{girthChoose} regarding the 3-choosability of every planar graph of girth at least 5 also works for signed graphs.

\begin{theorem} [{\cite{PlanarChoose}}]
	Every planar signed simple graph with neither 3-circuit nor 4-circuit is 3-choosable.
\end{theorem}

Recently, Kim and Yu \cite{Kim_Yu_2019} proved that every planar simple graph with no 4-circuits adjacent to 3-circuits is 4-choosable.

\section{Coloring generalized signed graphs} \label{Zhu}

Jin et al.~generalize the concepts of sections \ref{MRS} and \ref{KS} in \cite{JDY_Zhu_2019_Groetzsch, Jin_Zhu_S_labeled}. Before we establish this approach we need some further notation.   
In this context, a graph is considered as a symmetric digraph, where each edge $vw$ is replaced by two opposite arcs $e = (v,w)$ and
$e^{-1} = (w,v)$. Let $S$ be an inverse closed set of permutations of positive integers. An $S$-signature of $G$ is a  mapping
$\sigma : E(G) \rightarrow S$ such that $\sigma_{e^{-1}} = \sigma_e^{-1}$ for every arc $e$, and $(G,\sigma)$ is called an
$S$-signed graph. Let $[k] = \{1,2, \dots,k\}$. A $k$-coloring of $(G,\sigma)$ is a mapping $c: V(G) \rightarrow [k]$ 
such that $\sigma_e(c(v)) \not = c(w)$ for each arc $e=(v,w)$. The graph $G$ is $S$-$k$-colorable if $(G,\sigma)$ is $k$-colorable for every $S$-signature of $G$. 

The image of an integer $a \not \in [k]$ is irrelevant in an $S$-$k$-coloring of a graph $G$. Analogously, if $a \in [k]$ and
$\pi(a) \not \in [k]$, then $\pi(a)$ is irrelevant. In this sense, 
only permutations which are bijections between
subsets of $[k]$ are considered in \cite{JDY_Zhu_2019_Groetzsch, Jin_Zhu_S_labeled}.

$S$-$k$-coloring of graphs generalizes some further well known notions of coloring. We summarize: 

\begin{proposition} [\cite{JDY_Zhu_2019_Groetzsch}] Let $S$ be a subset of $S_k$.
\begin{itemize}
\item If $S = \{id\}$, then $S$-$k$-coloring is identical to conventional $k$-coloring.
\item If $S = \{ id, (1,2) (3,4) \dots ((2q-1),2q)$ and $q = \lfloor \frac{k}{2} \rfloor$, then $S$-$k$-coloring is 
	identical to $k$-coloring (Section \ref{MRS}).
\item If $S = \{ id, (1,2) (3,4) \dots ((2q'-1),2q')$ and $q' = \lceil \frac{k}{2} \rceil-1$, then $S$-$k$-coloring is 
	identical to modular $k$-coloring (Section \ref{KS}).
\item If $S = S_k$, then $S$-$k$-coloring is identical to DP-$k$-coloring, as defined in \cite{DP_Coloring_Def}.
\item If $S = <(1,2, \dots,k)>$ is the cyclic group generated by the permutation $(1,2, \dots,k)$, then
$S$-$k$-coloring is identical to $\mathbb {Z}_k$-coloring, as defined by Jaeger, Linial, Payan and Tarsi \cite{JLPT}.
\end{itemize}
\end{proposition}

Jin et al.~\cite{Jin_Zhu_S_labeled} also give an equivalent formulation of a $k$-coloring of a gain graph 
$(G,\phi)$ 
with gain group $\Gamma$, see \cite{Zaslavsky_gain_1995}, in terms of an $S$-$k'$-coloring of $(G,\sigma)$.
Coloring gain graphs is a separate topic which we will not discuss. Note that some facets of coloring gain graphs are already studied in \cite{Zaslavsky_Totally_Frustrated}.

Note that for $S = \{ id, (1,2) (3,4) \dots ((2q-1),2q)$ and $q = \lfloor \frac{k}{2} \rfloor$, 
a graph $G$ is $S$-$k$-colorable if and only if $k \geq M_{\chi_{\pm}}(G)$.
An analogous statement is true for modular $k$-coloring. 

The papers \cite{JDY_Zhu_2019_Groetzsch, Jin_Zhu_S_labeled, Kardos_etal_2019, Zhu_Note_2017, Zhu_non_Z_4_colorable}
focus on planar graphs. Two subsets $S$ and $S'$ of $S_k$ are conjugate if there is a permutation $\pi$ in $S_k$
such that $S'=\{\pi \sigma \pi^{-1}:  \sigma \in S\}$. 
A subset $S$ of $S_4$ is good if every planar graph is $S$-4-colorable, and it is bad if it is not good. 
Hence, by the 4-Color-Theorem, $\{id\}$ is good. 
M\'{a}\v{c}ajov\'{a}, Raspaud, and \v{S}koviera \cite{ModIntro} conjectured that 
$\{id, (1,2) (3,4)\}$ and Kang and Steffen conjectured that $\{id, (1,2)\}$ is good. Therefore, 
a natural question is whether there are subsets of $S_4$ which are good.
  
Jin et al.~\cite{Jin_Zhu_S_labeled} completely answer the aforementioned question for subsets $S$ of $S_4$ which contain $id$. 
They summarize: 
Kr\'{a}l et al.~\cite{Kral_etal_2005} showed that the set $\{id, (1234),(13)(24),(1432)\}$ as well as the set 
$\{id,(12)(34),(13)(24),(14)(23)\}$ are not good. 
In a first version of \cite{Jin_Zhu_S_labeled_1} Jin et al.~excluded $\{id,(123)\}$, $\{id, (1234)\}$, 
$\{id, (12),(13)\}$, $\{id, (12)(34),(13)\}$, $\{id, (12)(34), (13)(24)\}$. Zhu \cite{Zhu_refinement_choose} shows 
that $\{id,(12)\}$ is bad, and therefore, he disproved the conjecture of Kang and Steffen.
Kardo\v{s} and Narboni \cite{Kardos_etal_2019} disproved the conjecture of M\'{a}\v{c}ajov\'{a}, Raspaud, and \v{S}koviera 
by showing that $\{id,(12),(34)\}$ is bad. The remaining two cases for $S \in \{\{id, (123)\}, \{id, (1234)\}\}$ are
shown to be bad in the revised version \cite{Jin_Zhu_S_labeled} of \cite{Jin_Zhu_S_labeled_1} by Jin et al..
This completes the proof of the following theorem.

\begin{theorem}
A subset $S$ of $S_4$ is good if and only if $S = \{id\}$. 
\end{theorem}

Jiang et al.~\cite{JDY_Zhu_2019_Groetzsch} study the question whether 
Gr\"otzsch's Theorem can be generalized to generalized signed graphs. 
A non-empty subset $S$ of $S_3$ is $TFP$-good, if every triangle-free planar graph is $S$-3-colorable. 
Gr\"otzsch Theorem says that $\{id\}$ is $TFP$-good. 
They proved that an inverse closed subset of $S_3$ 
not isomorphic to $\{id,(12)\}$ is $TFP$-good if and only if $S=\{id\}$. Hence, the only remaining open case 
for this question is whether $\{id, (12)\}$ is $TFP$-good. 

Using the relation between signed graph coloring and DP-coloring, Kim and Ozeki \cite{Ozeki_2019}
gave a structural characterization of graphs that do not admit a DP-coloring.
This result generalizes Theorem \ref{char_Stiebitz}.

In a recent paper, Zhu \cite{Zhu_refinement_choose} refines the concept
of choosability such that the two extremal cases are $k$-choosability and $k$-colorability and 
in between there are gradually changing concepts of coloring which depend on the possible partitions
of the integer $k$.  In \cite{Zhu_refinement_choose} it is shown that several kinds of generalized 
signed graph coloring can be expressed in terms of a refined choosability concept. 

Further graph parameters, which are closely related to coloring parameters, are generalized to signed graphs. For instance,
Wang et al.~study the Alon-Tarsi number of signed graphs in \cite{Wang_etal_2019} and Lajou \cite{Lajou_2019} their achromatic number.

\section{Signed graph coloring via signed homomorphisms} \label{NRS_Hom}

Graph homomorphisms provide a unified language and useful tool for the study of graph coloring. 
Clearly, there are homomorphisms into appropriate target multigraphs for aforementioned coloring concepts for signed graphs and they are easy to define. 
For example, if $(G,\sigma)$ admits a $2k$-coloring, then there is a multigraph 
homomorphism from $(G,\sigma)$ to  $\pm K_{2k}^*$, where $\pm K_{2k}^*$ is 
obtained from $\pm K_{2k}$ by removing a perfect matching of negative
edges and adding a negative loop to each vertex. 

However, all approaches we have presented so far do not take
graph homomorphisms as starting point. This section changes the viewpoint and displays 
approaches to coloring signed graphs which have graph homomorphisms as starting point.
In the following, we consider graphs to be simple and loopless except when explicitly stated otherwise. 

In \cite{CircuitCover} B. Guenin introduced the notion of homomorphisms on signed graphs. This concept can be used to define a chromatic number for signed graphs which is different from those we have discussed so far. We will call this parameter the $h$-chromatic number of a signed graph. 

Homomorphism theory on signed graphs, in conformity with Zaslavsky's earlier statement, can be discussed in terms of switching classes of signed graphs. We will follow the definitions stated in \cite{Homomorphisms} and define homomorphisms on signed graphs as homomorphisms on switching classes of signed graphs. For the sake of simplicity, we will call both $(G,\sigma)$ and its corresponding switching class $[G,\sigma]$ signed graphs. The difference can always be spotted by looking at the brackets.

Given two signed graphs $[G,\sigma_1]$ and $[H,\sigma_2]$, we say that there is a homomorphism of $[G,\sigma_1]$ to $[H,\sigma_2]$ if there is a representative $(G,\sigma'_1)$ of $[G,\sigma_1]$ and a representative $(H,\sigma'_2)$ of $[H,\sigma_2]$ together with a vertex-mapping $\phi: V(G) \longrightarrow V(H)$ such that if $xy$ $\in E(G)$, then $\phi(x)\phi(y) \in E(H)$ and $\phi(x)\phi(y)$ has the same sign as $xy$. In other words, $\phi$ preserves signed adjacency. We will state the existence of a homomorphism of $[G,\sigma_1]$ to $[H,\sigma_2]$ as $ [G,\sigma_1] \longrightarrow [H,\sigma_2]$.

Let $\phi$ be a homomorphism of $[G,\sigma_1]$ to $[H,\sigma_2]$ using the representatives $(G,\sigma'_1)$ of $[G,\sigma_1]$ and $(H,\sigma'_2)$ of $[H,\sigma_2]$ and let $X \subseteq V(H)$ be the switching set that forms $(H,\sigma'_2)$ from $(H,\sigma_2)$. Let $(G,\sigma''_1)$ be the signed graph we obtain if we switch $(G,\sigma'_1)$ at $\phi^{-1}(X) \subseteq V(G)$. Then, $\phi$ is also a homomorphism of $[G,\sigma_1]$ to $[H,\sigma_2]$ using the representatives $(G,\sigma''_1)$ and $(H,\sigma_2)$. Therefore the exact representative of the image graph $[H,\sigma_2]$ is not important for the existence of a homomorphism. Note, however, that this does not hold for the domain graph since, for example, a signed forest admits a homomorphism to $[K_2,\texttt{\bf 1}]$ but the representative has to be either all-negative or all-positive.

An automorphism of $[G,\sigma]$ is a homomorphism of the signed graph to itself that is bijective on the vertex set $V(G)$ such that the induced edge-mapping is surjective. If, for each pair $x,y$ of vertices of $[G,\sigma]$, there exists an automorphism $\rho$ of $[G,\sigma]$ such that $\rho(x)=y$, $[G,\sigma]$ is called vertex-transitive. Similarly, if for each pair  $e_1=xy$ and $e_2= uv$ of edges of $[G,\sigma]$, there exists an automorphism $\rho$ of $[G,\sigma]$ such that $\{\rho(x),\rho(y)\}=\{u,v\}$, we call $[G,\sigma]$ edge-transitive. An unbalanced circuit is an example of a signed graph that is both vertex-transitive and edge-transitive since it has representatives with only one negative edge which can be moved around by switching one of its end vertices. Thus, there are automorphisms that induce a ``rotation'' on the circuit. We say that a signed graph $[G,\sigma]$ is isomorphic to $[H,\sigma']$ if there is a homomorphism of $[G,\sigma]$ to $[H,\sigma']$ that is bijective on the vertex set such that the induced edge-mapping is bijective as well.

A core of a signed graph $[G,\sigma]$ is a minimal subgraph of $[G,\sigma]$ to which $[G,\sigma]$ admits a homomorphism. A signed core is a signed graph that admits no homomorphism to a proper subgraph of itself, equivalently if every homomorphism of $[G,\sigma]$ to $[G,\sigma]$ is an automorphism, then $[G,\sigma]$ is a core. The following lemma shows that the core of a signed graph is well-defined.

\begin{lemma} [{\cite{Homomorphisms}}]
Let $[G,\sigma]$ be a signed graph. The core of $[G,\sigma]$ is unique up to isomorphism of signed graphs.
\end{lemma}

Since, in the context of homomorphisms, the signature of the image graph is of no concern, the binary relation of the existence of a homomorphism on signed graphs is transitiv.

\begin{lemma}
Let $[G,\sigma_1]$ and $[H,\sigma_2]$ be signed graphs. Then the relation $ [G,\sigma_1] \longrightarrow [H,\sigma_2]$ is transitiv.
\end{lemma}

Hence, the relation $ [G,\sigma_1] \longrightarrow [H,\sigma_2]$ is a quasi-order on the class of all signed graphs which is a poset on the class of all signed cores. 
Naserasr, Rollová and Sopena \cite{Homomorphisms} call this order the homomorphism order of signed graphs 
and say that $[H,\sigma_2]$ bounds $ [G,\sigma_1]$, or that $ [G,\sigma_1]$ is smaller than $[H,\sigma_2]$ 
instead of writing $ [G,\sigma_1] \longrightarrow [H,\sigma_2]$. Furthermore, they extend the notion to classes 
of graphs, so if $C$ is a class of signed graphs, they
 say that $[H,\sigma_2]$ bounds $C$ if $[H,\sigma_2]$ bounds every member of $C$. 

Now, as in the case of homomorphisms of graphs, the smallest order of a signed graph which bounds $[G,\sigma]$ defines a chromatic number of a signed graph. We call it the $h$-chromatic number
and denote it by $\chi_{Hom}([G,\sigma])$. 
Analogously, the notion of signed graph coloring can be defined in the following manner. A proper coloring of a signed graph 
$[G,\sigma]$ is an assignment of colors to the vertices of $G$ such that adjacent vertices do not receive the same color and there is a representative $(G,\sigma')$ of $[G,\sigma]$ such that any two edges whose end vertices receive the same two colors have
the same sign in $(G,\sigma')$. The $h$-chromatic number $\chi_{Hom}([G,\sigma])$ can then be regarded as the minimum number of colors needed for a proper coloring of $[G,\sigma]$. 

Note, that not all the different representatives $(G,\sigma')$ of $[G,\sigma]$ necessary admit a homomorphism to the bounding graph of $[G,\sigma]$ (since we know from above that the domain graph is not free in its representatives). Therefore, the $h$-chromatic number $\chi_{Hom}([G,\sigma])$ may also be expressed as a minimum over all representatives of $[G,\sigma]$: 
Let $Ord(G,\sigma)$ be the smallest order of a signed graph to which $(G,\sigma)$ 
admits a homomorphism. Then $$\chi_{Hom}([G,\sigma])=min\left\{Ord(G,\sigma):  (G,\sigma) \text{ is a representative of  }[G,\sigma] \right\}.$$
Note, that for an unbalanced circuit of length 4 we have $\chi_\pm((C_4,\sigma)) = 3$, $\chi((C_4,\sigma)) = 2$ and 
$\chi_{Hom}([C_4,\sigma]) = 4$. 

The question of the existence of a homomorphism between two signed graphs is essential for this definition of signed graph coloring. In this concern, the $h$-chromatic number itself provides a first test for the possibility of existence of a homomorphism of $[G,\sigma_1]$ to $[H,\sigma_2]$.

\begin{lemma} [{\cite{Homomorphisms}}]
Let $[G,\sigma_1]$ and $[H,\sigma_2]$ be signed graphs. If $[G,\sigma_1] \longrightarrow [H,\sigma_2]$, then $\chi_{Hom}([G,\sigma_1]) \leq \chi_{Hom}([H,\sigma_2])$.
\end{lemma}
This lemma follows from the transitivity of the relation of $ [G,\sigma_1] \longrightarrow [H,\sigma_2]$. There is a set of lemmas in \cite{Homomorphisms} that provide such  tests, they are called ``no homomorphism lemmas''.

In general, the problem ``Does a signed graph $[G,\sigma']$ admit a homomorphism to $[H,\sigma]$?'' is not easy to solve. For further results, we refer the interested reader to \cite{ComplexHom, ComplexHom_2, ComplexCycle}.

\subsection{Minor construction}

The construction of minors of signed graphs, as introduced in the context of signed graph homomorphism in \cite{Homomorphisms}, differs from the one  Zaslavsky proposed in \cite{SignedGraphs}. A minor of a signed graph $[G,\sigma]$ is a signed graph $[H,\sigma']$ obtained from $[G,\sigma]$ by a sequence of four operations: Deleting vertices, deleting edges, contracting positive edges and switching. While a negative edge can not directly be contracted, one can switch one of its end vertices and then contract it. Since in this paragraph we generally consider signed graphs to be simple, every contraction in this process that results in a multiple edge should be followed by the removal of all of these edges but one. Note that there is no rule concerning the choice of the remaining edge, so it can be chosen freely,
and the resulting graph depends on the choice.  
Using this rule also prevents the creation of loops in a minor.

The set of unbalanced circuits determines a signed graph, so it is worthwhile to consider that the contraction of positive edges does not change the sign of a circuit, hence we get the following lemma.

\begin{lemma} [\cite{Homomorphisms}]
Let $[H,\sigma']$ be a (simple) minor of a signed graph $[G,\sigma]$ that is obtained only by the operation of contracting positive edges, then the image of an unbalanced circuit of $[G,\sigma]$ is an unbalanced circuit in $[H,\sigma']$.
\end{lemma}

Furthermore, the above defined minor construction does not create circuits and thus, 
every minor of a balanced graph is balanced.

\subsection{Signed Cliques}

Another important extension to signed graphs is the notion of cliques. In \cite{Homomorphisms} the following generalization for signed graphs is introduced. Let $[G,\sigma]$ be a signed graph, $[G,\sigma]$ is a signed clique, or short S-clique, if its $h$-chromatic number equals its order. One may equivalently  call a signed graph an S-clique if its homomorphic images are all isomorphic to itself. Recall that for graphs, a clique is any complete graph, so an all-positive signed graph is an S-clique if and only if its underlying graph is a clique. Thus, this definition is a natural extension. Note that every signed complete graph $[K_n,\sigma]$ is an S-clique, but the converse is not true. A useful tool for determining whether a signed graph is an S-clique or not is the following lemma.

\begin{lemma} [{\cite{Homomorphisms}}]  \label{unbal4cycle}
A signed graph $[G,\sigma]$ is an S-clique if and only if for each pair $u$ and $v$ of vertices of $G$ either $uv$ is an edge in $G$ or $u$ and $u$ are vertices of an unbalanced circuit of length 4.
\end{lemma}

The following corollary follows from this lemma immediately.

\begin{corollary} [{\cite{Homomorphisms}}]
An S-clique cannot have a cut-vertex.
\end{corollary}

An interesting example of an S-clique that is not a signed complete graph is the signed complete bipartite graph $[K_{n,n},\sigma_M]$ and $N_{\sigma_M}$ is a perfect
matching. It is easy to see that $[K_{n,n},\sigma_M]$ is an S-clique for $n\geq 3$ if one checks for the existence of an unbalanced circuit of length 4 on every pair of vertices of the same partition.
Note that, since  the core of a signed graph is unique up to isomorphism, every S-clique is a core.

Following the definition of signed cliques, there are two natural definitions of the signed clique number of a signed graph. The absolute S-clique number of $[G,\sigma]$, denoted by $\omega_{sa}[G,\sigma]$, equals the order of the largest subgraph $[H,\sigma']$ of $[G,\sigma]$ that is an S-clique itself. The relative S-clique number of $[G,\sigma]$, denoted by $\omega_{sr}[G,\sigma]$, is the order of the largest subgraph $[H,\sigma']$ of $[G,\sigma]$ such that in every homomorphic image $\phi[G,\sigma]$ the order of the induced subgraph $\phi[H,\sigma']$ equals that of $[H,\sigma']$. Again it is easy to see that these definitions equal their unsigned counterpart in the case of an all-positive signed graph. It is also proved in \cite{Homomorphisms} that these definitions are independent of switching, so we can continue to consider switching classes instead of actual signed graphs. However, the difference between the absolute S-clique number and the relative S-clique number of a signed graph can be arbitrarily large. These numbers also follow the homomorphism order of signed graphs, which gives rise to another ``no homomorphism lemma''.

\begin{lemma} [{\cite{Homomorphisms}}]
Let $[G,\sigma_1]$ and  $[H,\sigma_2]$ be two signed graphs. If $[G,\sigma_1] \longrightarrow [H,\sigma_2]$, then $\omega_{sa}[G,\sigma_1] \leq \omega_{sa}[H,\sigma_2]$ and $\omega_{sr}[G,\sigma_1] \leq \omega_{sr}[H,\sigma_2]$.
\end{lemma}

Also, from the definition arises the following relationship between these two S-clique numbers and the $h$-chromatic number.

\begin{theorem} [{\cite{Homomorphisms}}]
Let $[G,\sigma]$ be a signed graph. Then $\omega_{sa}[G,\sigma] \leq \omega_{sr}[G,\sigma] \leq \chi_{Hom}([G,\sigma])$.
\end{theorem}

The signed complete bipartite graph $[K_{n,n},\sigma_M]$ is an example for the large differences that can occur between the $h$-chromatic number of a signed graph and the chromatic number of its underlying graph. It is bipartite and therefore, it is 2-colorable, 
but $\chi_{Hom}([K_{n,n},\sigma_M])=2n$. Hence, we have the following statement.

\begin{proposition} [{\cite{Homomorphisms}}]
For every graph $G$ and every signature $\sigma$ of $G$, $\chi_{Hom}([G,\sigma])\geq \chi (G)$. Furthermore, for each positive integer $n$, there is a a signed graph $[H,\sigma_H]$ such that $\chi_{Hom}([H,\sigma_H]) - \chi(H) \geq n$.
\end{proposition}

In \cite{Homomorphisms} it is shown that the problem of determining the relative or 
absolute S-clique number of a signed graph is NP-hard.
However, in \cite{RelCNr} the relative S-clique number is determined for graphs of some families 
of planar and outerplanar signed graphs.

Now that the basic definitions necessary for signed graph coloring in the context of homomorphisms are introduced, we continue by stating some of their elemental properties. We know that there is (up to switching) only one signed graph on any given tree and 
its core is always the $[K_2,\texttt{\bf 1}]$. Therefore we get the following proposition regarding the S-clique numbers on a signed tree.

\begin{proposition} [{\cite{Homomorphisms}}]
Let $[G,\sigma]$ be a signed graph. If $G$ is a tree, 
then $\omega_{sa}[G,\sigma]=\omega_{sr}[G,\sigma]=\chi_{Hom}([G,\sigma])=2$.
\end{proposition}

Furthermore, the set of (homomorphically) 2-colorable signed graphs can be completely classified by two properties.

\begin{theorem} [{\cite{Homomorphisms}}]
A signed graph $[G,\sigma]$ is (homomorphically) 2-colorable if and only if $G$ is bipartite and $[G,\sigma]$ is balanced.
\end{theorem}

For other values than 2 however, it is difficult to compute the $h$-chromatic number. In \cite{Homomorphisms} it is shown that
the problem ``Is $\chi_{Hom}([G,\sigma]) \leq k$?'' is computable in polynomial-time for $k \in \{1,2\}$ and it is NP-complete for $k \geq 3$.

Motivated by Brooks' Theorem, the relation between the $h$-chromatic number and the maximum degree of a graph is studied.

\begin{theorem} [{\cite{MaxDegree}}]
For every signed graph $[G,\sigma]$ with $\Delta(G) \geq 3$ holds: 
 
\[2^{\frac{\Delta(G)}{2}-1 }\leq \chi_{Hom}([G,\sigma]) \leq (\Delta(G) - 1)^2 2^{(\Delta(G)-1)} + 2.\]
\end{theorem}

There are further upper bounds for the $h$-chromatic number in terms of homomorphisms to some target graphs which are studied by Ochem, Pinlou and Sen in \cite{Ochem, ochem2017homomorphisms} and by Naserasr, Rollová and Sopena in \cite{HomomorphismCubes}. The latter three authors studied minor closed 
families of signed graphs and achieved the following results on the $h$-chromatic number. 

\begin{theorem}[\cite{Homomorphisms}]
Let $G$ be a $K_4$-minor-free graph. If $[G,\sigma]$ is a signed graph, then $\chi_{Hom}([G,\sigma]) \leq 5$, and this bound is tight. 
\end{theorem}

This theorem implies that in particular the $h$-chromatic number of an outerplanar signed graph $[G,\sigma]$ is at most 5. It is
shown in \cite{Homomorphisms} that the bound is tight for this class as well.

It is proved in \cite{Coexter}, that every $m$-edge-colored graph whose underlying graph has an acyclic chromatic number of at most $k$ admits a homomorphism to an $m$-edge-colored graph of order at most $km^{k-1}$. This result is generalized to colored mixed graphs in \cite{ColoredMixed} and there is a version of this in \cite{Homomorphisms} for the case of signed graphs as well. Thus, 
giving a bound for the class of signed graphs whose underlying graphs are acyclically $k$-colorable.

\begin{theorem} [{\cite{Ochem}}] \label{acyclically}
Let $[G,\sigma]$ be a signed graph. If $G$ is acyclically $k$-colorable, then $\chi_{Hom}([G,\sigma]) \leq k 2^{k-2}$.
\end{theorem}

Ochem and Pinlou note that the bound of Theorem \ref{acyclically} is tight, 
which is shown in \cite{Fabila}. 
Theorem \ref{acyclically} can be used to give a more general rule regarding the upper bounds of the $h$-chromatic number of some classes of signed graphs. Recall that a $k$-tree is a graph that can be constructed from the complete graph $K_k$ by repeatedly adding vertices in such a way that each added vertex is joined to $k$ vertices that already form a $k$-clique. A subgraph of a $k$-tree is called a partial $k$-tree. All $k$-trees are acyclically $(k+1)$-colorable and thus,
following the last theorem we get:

\begin{corollary}
Let $G$ be a partial $k$-tree and let $\sigma$ be any signature on $G$. 
Then $\chi_{Hom}([G,\sigma]) \leq (k+1)2^{k-1}$.
\end{corollary}

In \cite{SteinerTrees}, it is proved that $K_4$-minor-free graphs are exactly partial 2-trees and for these the above formula only gives an upper bound of 6 instead of the earlier mentioned bound of 5, so its bounds are generally not tight. 
There is also an upper bound on the $h$-chromatic number of the class of planar signed graphs that can be obtained by using the bound on the acyclic chromatic number of planar graphs and techniques similar to the ones applied in \cite{GoodColorings} and \cite{Coexter}, equivalently from Theorem \ref{acyclically} and the fact that every planar graph is acyclically 5-colorable (see \cite{AcyclicColorings}). 
The following theorem seems to be proved parallel in \cite{Homomorphisms} and \cite{Ochem}. In \cite{Homomorphisms} it
is proved for 48 instead of 40, but as remarked in \cite{Homomorphisms} using Theorem \ref{acyclically} yields the following 
the statement.

\begin{theorem} [\cite{Homomorphisms, Ochem}] \label{planar_Naser_Ochem}
Let $[G,\sigma]$ be a planar signed graph. Then $\chi_{Hom}([G,\sigma]) \leq 40$. Also, there are a planar S-clique of order 8 and a planar signed graph with $h$-chromatic number 10.
\end{theorem}

Let $\mathcal{P}_g$ ($\mathcal{O}_g$) denote the class of planar (outerplanar) 
signed graphs with girth at least $g$. 

\begin{theorem} [{\cite{Ochem}}]
Let $[G,\sigma]$ be a planar signed graph.
\renewcommand{\labelenumi}{(\roman{enumi})}
\begin{enumerate} 
\item If $[G,\sigma] \in \mathcal{O}_4$, then $\chi_{Hom}([G,\sigma]) \leq 4$
\item If $[G,\sigma] \in \mathcal{P}_4$, then $\chi_{Hom}([G,\sigma]) \leq 25$
\item If $[G,\sigma] \in \mathcal{P}_5$, then $\chi_{Hom}([G,\sigma]) \leq 10$
\item If $[G,\sigma] \in \mathcal{P}_6$, then $\chi_{Hom}([G,\sigma]) \leq 6$
\end{enumerate}
\end{theorem}

In \cite{ochem2017homomorphisms}, the authors consider 2-edge-colored graphs instead of signed
graphs and deduce the same results in this context. In \cite{naserasr2019} the relation between these two
approaches are studied in the context of (signed) graph homomorphisms. Further results on
the $h$-chromatic number of grids are obtained in \cite{DYBIZBANSKI2020105918}.

An interesting fact is that if the maximum $h$-chromatic number of any planar signed graph were $k$, then this would imply the existence of a signed graph of order $k$ to which every planar signed graph admits a homomorphism. There are important results for other properties of planar signed graphs as well. For example the maximum order of a planar S-clique and therefore, the maximum of the absolute S-clique number of planar graphs.
 \begin{theorem} [\cite{Homomorphisms}]
The maximum order of a planar S-clique is 8.
\end{theorem}
The proof of this theorem heavily relies on Lemma \ref{unbal4cycle}, indicating it as a useful tool regarding S-cliques.
Furthermore, the relations between graphs and signed bipartite graphs can be used to restate Hadwiger's conjecture and in the following, leads to possibilities of a strengthening of Hadwiger's conjecture for the class of even signed graphs. This topic is extensively
studied in \cite{Homomorphisms}.

\section{Final remarks and some conjectures}

A definition of a ``chromatic number'' for signed graphs strongly depends on properties of the colors, 
as those of the ``signed colors'' in the definitions 
of Zaslavsky and M\'a\v{c}ajov\'a et al.~or on the permutations which are associated to the edges as in the case of 
generalized signed graphs. 
Since every element of an additive abelian group has an inverse element, the condition $c(v) \not= \sigma(vw) c(w)$ is
equal to the condition $c(v) \not = c(w)$ if the color $c(v)$ is self-inverse; i.e.~$\sigma$ is the identity on $c(v)$. 
The self-inverse elements play a crucial role in such colorings, since the color classes which are induced by these elements are independent sets. Hence, the following statement is true.

\begin{proposition} Let $G$ be a  graph and $\chi(G) = k$. If ${\cal C}$ is a set of $k$ pairwise different self-inverse 
colors (e.g.~of $\mathbb{Z}_2^n$ $(k \leq 2^n)$), then every $k$-coloring of $G$ with colors from ${\cal C}$  is a $k$-coloring of $(G,\sigma)$, for every signature $\sigma$ of $G$. In particular, the chromatic number of $(G,\sigma)$ with respect to colorings with colors of ${\cal C}$ is equal to the chromatic number of $G$ for every signature $\sigma$. 
\end{proposition}

A graph $G$ together with a function $f : V(G) \longrightarrow \{\pm 1 \}$ is a marked graph. This naturally induces a 
signature on $G$, where an edge is positive if its two vertices have the same mark, and it is negative
otherwise. Harary and Kabell \cite{Harary_Kabell} noticed that the signed graph where the signature is obtained 
from a marking of the vertices is balanced.  
This fact implies that the edge-chromatic number of a signed graph $(G,\sigma)$ is equal to edge-chromatic number of
$G$, which is noticed by Schweser and Stiebitz in \cite{choosable}. 

Despite of the early approached of Cartwright and Harary \cite{C_H_1968}, almost all concepts 
of signed graph coloring and their chromatic numbers  
are natural generalizations of the corresponding concepts for graphs. 
The next problem is only of interest for signed graphs.

\begin{problem} \label{P_Complexity} What is the complexity of the following decision problem: 
Let $k$ and $t$ be integers and $(G,\sigma)$ be a signed graph and $\chi(G) = k$.
Is $\chi((G,\sigma)) \leq t$ $(\chi_{\pm}((G,\sigma)) \leq t)$?
\end{problem}

It is easy to figure out the trivial cases, e.g.~if $k \in \{1,2\}$ or $t \geq \min\{\Delta(G), 2k-2\}$.
Brewster et al.~\cite{ComplexHom} proved that it can be decided in polynomial time whether two signed graphs are equivalent.

In \cite{Spectrum} (see Section \ref{spectrum}) it is shown that the chromatic spectrum is an interval of integers for $n$- and for modular $n$-colorings.  
Let $(G,\sigma)$ be an $S$-signed graph and $c_S((G,\sigma))$ be the minimum $k$ such that $(G,\sigma)$ has a $k$-coloring (for this fixed labelling of the arcs). 
A graph $G$ is $S$-$k$-colorable if  
$k \geq \max \{c_S((G,\sigma)) : \sigma \text{ is an } S\text{-signature of } G\}$.

\begin{problem} \label{Probl_S_chromatic_spectrum} 
Let $S$ be an inverse closed set of permutations of $S_k$ and $G$ be $S$-$k$-colorable.
Is it true that the set 
$\{c_S((G,\sigma)) : \sigma \text{ is an } S\text{-signature of } G\}$ is an interval of integers? 
\end{problem}

In the context of coloring planar graphs, the following questions might be of interest. 
Problem \ref{P_order} is formulated for 
$n$-coloring of signed graphs (see Section \ref{MRS}) by Kardo\v{s} and Narboni \cite{Kardos_etal_2019}.

\begin{problem} \label{P_order} Let $S \subseteq S_4$ and $id \in S$. What is the smallest order of a non-$S$-4-colorable planar graph?
\end{problem}

Note that for odd $k$, a signed graph is $k$-colorable if and only if it is modular $k$-colorable. Hence
the following conjecture also applies for modular coloring. 

\begin{conjecture} [\cite{JDY_Zhu_2019_Groetzsch}]
Every triangle-free planar signed graph is 3-colorable. 
\end{conjecture}

Theorem \ref{relation_cc} says that the difference between the circular chromatic number and the chromatic number can be 1.  

\begin{problem}
Is it true that every planar signed graph has circular chromatic number at most 4?
\end{problem}

\subsubsection*{Acknowledgement}
We thank Thomas Zaslavsky for very careful reading the manuscript and many very helpful remarks and corrections which improve the paper a lot. 
We also thank two anonymous referees for their valuable comments. One of them pointed out that there is 
a short survey on some aspects of coloring signed graphs by Lynn Takeshita \cite{Lynn}.

\bibliography{Survey_References}{}
\addcontentsline{toc}{section}{References}
\bibliographystyle{abbrv}
\end{document}